\documentclass[11pt,notitlepage,twoside,a4paper]{amsart}
 \usepackage{amsfonts}

\usepackage{amsmath,amssymb,enumerate}

\usepackage{epsfig}

\usepackage{amssymb}
\usepackage{amsmath,amsthm}
\usepackage{latexsym}
\usepackage{amscd}
\usepackage{psfrag}
\usepackage{graphicx}
\usepackage[latin1]{inputenc}
\usepackage[all]{xy}
\usepackage[mathcal]{eucal}

\renewcommand{\textsc}{\textcolor{red}}

\newtheorem{theorem}{\rm\bf Theorem}[section]
\newtheorem{proposition}[theorem]{\rm\bf Proposition}
\newtheorem{lemma}[theorem]{\rm\bf Lemma}
\newtheorem{corollary}[theorem]{\rm\bf Corollary}
\newtheorem*{theorem 1}{\rm\bf Proposition 1}
\newtheorem*{theorem 2}{\rm\bf Proposition 2}

\theoremstyle{definition}

\theoremstyle{remark}

\def\interieur#1{\mathord{\mathop{\kern 0pt #1}\limits^\circ}}

\title[The ideal triangulation graph]{On the ideal triangulation graph of a punctured surface}

\author{Mustafa Korkmaz}
\address{Mustafa Korkmaz, Department of Mathematics, Middle East Technical University, 06531 Ankara, Turkey.}
\email{korkmaz@metu.edu.tr}

\author{Athanase Papadopoulos}
\address{Athanase Papadopoulos, Max-Plank-Institut f\"ur Mathematik, Vivatsgasse 7, 53111 Bonn,
Germany, and: Institut de Recherche Math{\'e}matique Avanc\'ee,
Universit{\'e} de Strasbourg and CNRS,
7 rue Ren\'e Descartes,
 67084 Strasbourg Cedex, France.} \email{papadopoulos@math.u-strasbg.fr}

\date{\today}

\begin{document}

\begin{abstract}
We study the ideal triangulation graph $T(S)$ of a punctured
surface $S$ of finite type. We show that if $S$ is not the sphere
with at most three punctures or the torus with one puncture, then
the natural map from the extended mapping class group of $S$ into
the simplicial automorphism group of  $T(S)$  is an isomorphism.
We also show that under the same 
conditions on $S$,
 the graph $T(S)$ equipped with its natural simplicial metric is
not Gromov hyperbolic. Thus, from the point of view of Gromov hyperbolicity, the situation of $T(S)$ is different from that of the curve complex of $S$.
\bigskip


\medskip


\end{abstract}
\maketitle

\section{Introduction}
\label{intro}

In this paper, $S$ is a connected orientable surface of finite type, of genus
$g\geq 0$ without boundary and with $n\geq 1$ punctures. We shall
assume that the Euler characteristic $\chi(S)$ of $S$ is negative.
The {\it mapping class group of $S$}, denoted by
$\mathrm{Mod}(S)$, is the group of isotopy classes of
orientation-preserving homeomorphisms of $S$. The {\it extended
mapping class group of $S$}, $\mathrm{Mod}^*(S)$, is the group of
isotopy classes of all homeomorphisms of $S$.

We denote by $\overline{S}$ the surface obtained from $S$ by
filling in the punctures. Thus, $\overline{S}$ is a closed surface
of genus $g$. The punctures can also be considered as
distinguished points on $\overline{S}$, and we denote by $B\subset
\overline{S}$ this set of distinguished points. An {\it arc} in
$S$ (or in $\overline{S}$) is the image in $\overline{S}$ of a
closed interval whose interior is homeomorphically embedded in
$\overline{S}\setminus B$ and whose endpoints are on $B$. An arc
in $S$ (or in $\overline{S}$) is said to be {\it essential} if it
is not homotopic (relative to $B$) to a point in $\overline{S}$.

 All homotopies of
arcs considered in this paper are relative to their endpoints.
An {\it ideal triangulation} (or, for short, a {\it
triangulation}) of $S$ is a maximal collection of disjoint
essential arcs that are pairwise non-homotopic. An
essential arc that belongs to a triangulation will also be called
an {\it edge} of that triangulation. A triangulation will
sometimes be identified with the union of its edges, and will
therefore be considered as a subset of $S$. A {\it face} of a
triangulation $\Delta$ is a connected component of the complement
of $\Delta$ in $S$ (or the closure of such a component), and it
will also be called a {\it triangle}.

An {\it elementary move} is the operation of obtaining a
triangulation from a given one by removing an edge and replacing
it by a distinct edge. If $a$ is the edge that is removed, then we
shall say that the elementary move is performed on $a$.

\begin{figure}[hbp]
\psfrag{a}{$a$}
 \psfrag{b}{$a^*$}
\centering
\includegraphics[width=.4\linewidth]{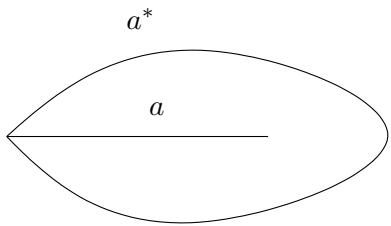}
\caption{\small{The edge $a$ is a non-exchangeable arc in
a triangulation, and $a^*$ is its dual edge.}}
\label{fig:isolated}
\end{figure}

An edge $a$ in a triangulation $\Delta$ is said to be {\it
exchangeable} if one can perform on $\Delta$ an elementary move,
obtaining a new triangulation in which $a$ is replaced by an edge
distinct from $a$. Otherwise, the edge is said to be {\it
non-exchangeable}. It follows from the classification of surfaces
that an edge in a triangulation $\Delta$ is non-exchangeable if
and only if it joins two distinct punctures and if this edge is in
the closure of a unique face of $\Delta$. In this case, there is a
well-defined edge $a^*$ in $\Delta$ which is associated to $a$,
which we call the edge {\it dual to} $a$, and which is the third
edge of the unique triangle of $\Delta$ to which $a$ belongs. This
situation is described in Figure \ref{fig:isolated}.

In this paper, we study the  {\it ideal triangulation graph} of
$S$. This is the simplicial graph $T(S)$ whose vertices are
isotopy classes of triangulations of $S$ in which an edge connects
two vertices whenever these two vertices differ by an elementary
move. The extended mapping class group $\mathrm{Mod}^*(S)$ acts
naturally on $T(S)$ by simplicial automorphisms. We prove the
following.

\begin{theorem}\label{th:automorphism}
Let $S$ be a connected orientable surface with at least one
puncture. If $S$ is not a sphere with at most three punctures or a
torus with one puncture, then the natural homomorphism
$\mathrm{Mod}^*(S)\to \mathrm{Aut}(T(S)$ is an isomorphism.
\end{theorem}

Note that the hypothesis on $S$ made in this theorem, combined
with the condition   $\chi(S)\leq -1$, is equivalent to
$\chi(S)\leq -2$.

The proof of Theorem \ref{th:automorphism} involves the
consideration of the arc complex of $S$.  This is the abstract
simplicial complex, denoted by $A(S)$, whose $k$-simplices, for
each $k\geq 0$, are the collections of $k+1$ distinct isotopy
classes of essential arcs on $S$ which can be represented by
pairwise disjoint arcs on this surface.  The ideal triangulation
graph has been studied in \cite{Harer} and \cite{Hatcher}. The arc
complex has been studied in  \cite{ivanov}, \cite{korkmaz} and
\cite{IM}. The idea of the proof of Theorem \ref{th:automorphism}
follows a general scheme that was used in \cite{IK} and
\cite{Margalit}.

In the proof of Theorem \ref{th:automorphism}, we shall use the
following analogue of that theorem for the arc  complex, which was
obtained by Irmak and McCarthy:

\begin{theorem} $($\cite{IM}$)$ \label{th:automorphism-arc}
Let $S$ be a connected orientable surface with at least one
puncture. If the surface $S$ is not a sphere with at most three
punctures or a torus with one puncture, then the natural
homomorphism $\mathrm{Mod}^*(S)\to \mathrm{Aut}(A(S))$ is an
isomorphism.
\end{theorem}

The graph $T(S)$ is the one-skeleton of the simplicial complex
dual to the arc complex $A(S)$. Thus, any automorphism of $A(S)$
induces an automorphism of the graph $T(S)$. But since $T(S)$ is a
strict subcomplex of the dual complex of $A(S)$, its automorphism group
could a priori be larger than the automorphism group of $A(S)$.
Theorem~\ref{th:automorphism} shows that this is not the case.

There are three special cases of surfaces that admit ideal
triangulations and that are excluded by the hypothesis of Theorem
\ref{th:automorphism}: the cases of a sphere with two or three
punctures and the case of a torus with one puncture. Let us
briefly discuss these cases.

If $S$ is a sphere with two punctures, then $T(S)$ consists of
only one vertex, hence its automorphism group is trivial. But the
extended mapping class group $\mathrm{Mod}^*(S)$ is isomorphic to
Klein's four group, the direct sum of two cyclic groups of order
two. Thus, the
natural homomorphism $\mathrm{Mod}^*(S)\to \mathrm{Aut}(T(S))$ is
surjective and not injective.

If $S$ is a sphere with three punctures, then any ideal
triangulation of $S$ has three vertices, three edges and two
faces. The simplicial complex $T(S)$ is a finite graph,
homeomorphic to a tripod. The center of the tripod corresponds to
the ideal triangulation of $S$ in which every edge is
exchangeable, and the three other vertices of $T(S)$ correspond to
ideal triangulations in which exactly one edge is exchangeable.
The automorphism group $\mathrm{Aut}(T(S))$ of $T(S)$ is
isomorphic to the permutation group on three elements (the three
vertices of valency one of the tripod). The mapping class group of
$S$ is also isomorphic to the permutation group on three element
(the three punctures of $S$), and the natural homomorphism
$\mathrm{Mod}(S)\to \mathrm{Aut}(T(S))$ is an isomorphism. The
natural homomorphism $\mathrm{Mod}^*(S)\to \mathrm{Aut}(T(S))$ is
surjective and its kernel is the center of $\mathrm{Mod}^*(S)$,
which is a cyclic group of order two.

If $S$ is a torus with one puncture, then the ideal triangulation
graph $T(S)$ is a regular infinite tree in which every vertex has
valency $3$. The automorphism group of such a tree is uncountable.
To see this, we consider the set of one-sided infinite sequences
of letters on the alphabet $\{f,n\}$. This set of sequences is
uncountable, and it can be injected in the simplicial automorphism
group of the tree. Such an injection can be done by choosing a
base point of the tree, then embedding the set of one-sides
sequences as the set of simplicial geodesic rays starting at that
base point, and finally defining for each such geodesic ray an
automorphisms of the tree by interpreting the letter $f$ as a
flip, and the letter $n$ as no flip. Thus, in the case considered,
the natural homomorphism $\mathrm{Mod}^*(S)\to \mathrm{Aut}(T(S))$
is highly non-surjective since the extended mapping class group
$\mathrm{Mod}^*(S)$ is countable.

By declaring that the length of each edge is one, we may
consider $T(S)$ as a metric space. With this metric, the
\textit{length} of a simplicial path between two vertices is the
number of edges in that path.
It is now natural to ask
whether $T(S)$ is Gromov-hyperbolic.

Theorem  \ref{th:automorphism} implies the following:

\begin{theorem}\label{not-hyp}
Let $S$ be a connected orientable surface with at least one
puncture. Suppose that $S$ is not a sphere with at most three
punctures or a torus with one puncture. Then the ideal
triangulation graph of $S$ is not Gromov-hyperbolic.
\end{theorem}

\begin{proof}
The action of the extended mapping class group $\mathrm{Mod}^*(S)$
on the vertices of $T(S)$ is free, since a  homeomorphism of $S$
which fixes the homotopy class of an ideal triangulation is
homotopic to the identity. The action of $\mathrm{Mod}^*(S)$ on
$T(S)$ is co-compact, since up to homeomorphisms,  there are only
finitely homotopy classes of ideal triangulations on any surface
of finite type. This is because the surface $S$ equipped with any ideal triangulation is
obtained by gluing $4g-4+2n$ triangles along edges, 
and there are only finitely many
ways to glue such a finite set of triangles to get a triangulated surface. Thus, if $T(S)$ were
Gromov-hyperbolic, then $\mathrm{Mod}^*(S)$ would be
word-hyperbolic (see \cite{Gromov}). But under the hypotheses of
the theorem, the extended mapping class group $\mathrm{Mod}^*(S)$
is not word-hyperbolic, since it contains free abelian groups of
rank $\geq 2$.
 \end{proof}

We note that in  the cases excluded in the hypothesis of Corollary
\ref{not-hyp}, the situations are as follows:

  \begin{enumerate}[$\bullet$]
  \item If $S$ is a sphere with one puncture, then the triangulation
graph $T(S)$ is empty. 
\item If $S$ is a sphere with two punctures, then
$T(S)$ consists of only one vertex, hence it is hyperbolic.

\item    If $S$ is a sphere with three punctures, then
its triangulation graph is compact, and therefore hyperbolic.

\item  If $S$ is a torus with one puncture, then
its triangulation graph is hyperbolic since, as we recalled above,
it is a tree.
\end{enumerate}

  \medskip

Theorem \ref{not-hyp} says that the large-scale geometry of the
triangulation graph of $S$ is different from that of the curve
complex, which, by a result of Masur and Minsky, is hyperbolic.

\section{Squares and pentagons in $T(S)$}

To simplify notation, we shall often identify an arc or a
triangulation on $S$ with its isotopy class.

We shall use two special classes of simplicial closed paths in
$T(S)$, and we now describe them.

A {\it square} in $T(S)$ is a simple closed (that is, injective)
path in this graph consisting of four edges,  as represented in
Figure \ref{fig:square}. In this figure, elementary moves are
performed on exchangeable edges $a$ and $c$, whose images under
these moves are denoted, respectively, by $a'$ and $c'$. Note that
the existence of such an elementary move implies that the
interiors of the faces containing $a$ and $c$ in the surface $S$
are disjoint. The move is also represented diagrammatically as
follows:

\[\langle a,c\rangle \leftrightarrow\langle a,c'\rangle \leftrightarrow
\langle a',c'\rangle \leftrightarrow\langle a',c\rangle \leftrightarrow
\langle a,c\rangle .
\]

In this notation, a symbol such as $\langle a,c\rangle
\leftrightarrow\langle a,c'\rangle$ represents an elementary move
between two triangulations, the first one having $a$ and $c$ among
its edges and the second one having $a$ and $c'$ among its edges,
and where the move is performed on $c$, which is transformed into
$c'$. All other edges remain unchanged.

\begin{figure}[hbp]
 \psfrag{a}{$a$}
\psfrag{b}{$a'$}
\psfrag{c}{$c$}
\psfrag{d}{$c'$}
\centering
\includegraphics[width=.5\linewidth]{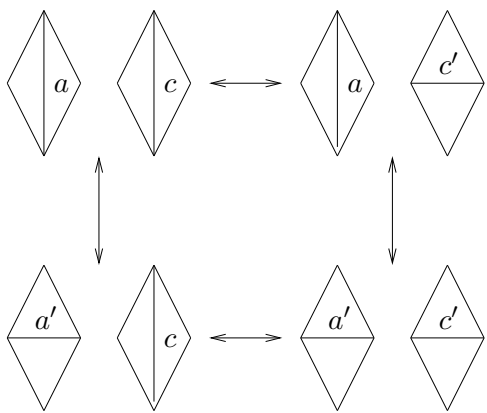}
\caption{\small{A square in the triangulation graph.}}
\label{fig:square}
\end{figure}

A {\it pentagon} in $T(S)$ is a simple closed path in $T(S)$ of
length five, represented diagrammatically in Figure
\ref{fig:pentagon}. In that figure, at each vertex, we have
indicated the name of a pair of edges, one which remains invariant
after a move represented by an edge adjacent to that vertex, and
one which is transformed by that move into an edge whose name
appears in the label of the corresponding adjacent vertex.  This
move is also represented diagrammatically as follows:

\[\langle a,b\rangle \leftrightarrow\langle a,e\rangle
\leftrightarrow\langle e,d\rangle \leftrightarrow\langle d,c\rangle
\leftrightarrow\langle c,b\rangle \leftrightarrow\langle b,a\rangle.
\]

\begin{figure}[hbp]
   \psfrag{A}{$\langle a,b\rangle $}
\psfrag{B}{$\langle a,e\rangle $}
\psfrag{C}{$\langle e,d\rangle $}
\psfrag{D}{$\langle d,c\rangle $}
\psfrag{E}{$\langle c,b\rangle $}
\centering
\includegraphics[width=.3\linewidth]{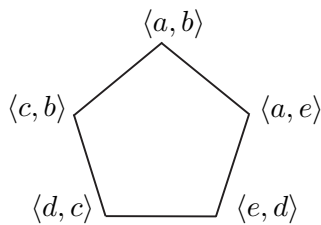}
\caption{\small{A pentagon in the triangulation graph.}}
\label{fig:pentagon}
\end{figure}

 \begin{lemma}\label{lem:triangle}
 There are no closed paths of length three in $T(S)$.
 \end{lemma}
  \begin{proof}
 Assume there exists such a path and let us represent it by a
 diagram
 \[\langle a,b\rangle \leftrightarrow\langle b,c\rangle \leftrightarrow
 \langle c,a\rangle \leftrightarrow\langle a,b\rangle
 \]
in which $a$ and $c$ are distinct arcs and where the double arrows
represent, as before, elementary moves, that is, the arc $a$ is
transformed into the arc $c$ by the first move, and so on. But the
existence of the moves $\langle a,b\rangle \leftrightarrow\langle
b,c\rangle $ and  $\langle a,b\rangle \leftrightarrow\langle
c,a\rangle $ imply $a=b$, a contradiction. This proves the lemma.
  \end{proof}

 \begin{lemma}\label{lem:square}
Every simple closed path of length four in $T(S)$ is a square
(that is, it is of the form described in Figure
 \ref{fig:square}).
 \end{lemma}
\begin{proof} Consider a simple closed path of length four in  $T(S)$, and let us  label it as
\[\langle a,b\rangle \leftrightarrow\langle a,c\rangle \leftrightarrow\langle c,d\rangle
\leftrightarrow\langle d,e\rangle \leftrightarrow\langle a,b\rangle .\]
Then, the existence of the last move implies that either $d=a$, or
$d=b$, or $e=a$, or $e=b$. We analyze each case separately.

  \begin{enumerate}[$\bullet$]
 \item The case $d=a$ is excluded, because of the existence of the move
 $\langle a,c\rangle \leftrightarrow\langle c,d\rangle$.

 \item If $d=b$ then we get two moves
 $ \langle b,e\rangle \leftrightarrow\langle a,b\rangle$ and
 $\langle c,b\rangle\leftrightarrow\langle b,e\rangle$ which give $a=c$,
 which is excluded since we have a vertex labelled $\langle a,c\rangle$ in the path.

\item  If $e=a$ then we get two moves  $\langle a,b\rangle\leftrightarrow\langle a,c\rangle$ and
   $ \langle a,d\rangle \leftrightarrow\langle a,b\rangle$ which gives $d=c$,
   which is excluded since we have a vertex labelled $\langle c,d\rangle$ in the path.

 \item  The remaining case is $e=b$, which labels the simple closed path in the following way:
    \[\langle a,b\rangle\leftrightarrow\langle a,c\rangle\leftrightarrow\langle
    c,d\rangle\leftrightarrow\langle d,b\rangle\leftrightarrow\langle a,b\rangle,\]
    which is of the form represented in Figure \ref{fig:square}.

  \end{enumerate}
  \end{proof}

 \begin{lemma}\label{lem:pentagon}
 Every simple closed path of length five in $T(S)$ is a pentagon
 (that is, it is of the form described in Figure \ref{fig:pentagon}).
 \end{lemma}

\begin{proof} The proof is by inspection, like the proof of
Lemma \ref{lem:square} (and it uses Lemma \ref{lem:triangle}).
\end{proof}

In the next two lemmas, and later in the paper, if $\Delta$ and
$\Delta'$ are two triangulations of $S$ connected by an edge in
$T(S)$, then the notation $\Delta-\Delta'$ will be used to denote
(the isotopy classes of) the edge in $\Delta$ that is not in
$\Delta'$.

 \begin{lemma}\label{lem:square1}
 Consider a square in $T(S)$, represented as
 $\Delta_1\leftrightarrow \Delta_2\leftrightarrow \Delta_3\leftrightarrow
 \Delta_4 \leftrightarrow \Delta_1$. Then we have $\Delta_1-\Delta_4=\Delta_2-\Delta_3$.
 \end{lemma}

\begin{proof}
This can be checked on Figure \ref{fig:square}.
\end{proof}

 \begin{lemma}\label{lem:pentagon1}
Consider a pentagon in $T(S)$, represented as
$\Delta_1\leftrightarrow  \Delta_2\leftrightarrow \Delta_3
\leftrightarrow  \Delta_4 \leftrightarrow  \Delta_5
\leftrightarrow  \Delta_1$. Then we have
$\Delta_1-\Delta_5=\Delta_2-\Delta_3$.
 \end{lemma}

\begin{proof}
This can be checked on Figure \ref{fig:pentagon}.
\end{proof}

 \section{Proof of Theorem \ref{th:automorphism}}
Let $f:T(S)\to T(S)$ be a simplicial automorphism. We associate to
$f$ a simplicial automorphism $\widetilde{f}:A(S)\to A(S)$,
defined as follows.

Let $a$ be an essential arc on $S$. We choose a triangulation
$\Delta$ in which $a$ is an exchangeable edge. Such a
triangulation always exists. Since $a$ is  exchangeable in
$\Delta$, we can perform an elementary move on $a$, replacing it
by an arc $a'$. Let $\Delta_a$ be the triangulation obtained from
$\Delta$ by this elementary move. Since $f$ is simplicial, the
triangulations $f(\Delta)$ and $f(\Delta_a)$ (like $\Delta$ and
$\Delta_a$) are joined by an edge in $T(S)$. In other words, the
two triangulations $f(\Delta)$ and $f(\Delta_a)$ differ by an
elementary move. We then define $\widetilde{f}(a)$ to be the edge
that is in $f(\Delta)$ but not in $f(\Delta_a)$. In short, we have
\[\widetilde{f}(a)=f(\Delta) - f( \Delta_a ).\]

 We now prove the following:

\begin{proposition} \label{prop:well}
The map $\widetilde{f}:A(S)\to A(S)$ is well-defined. That is, for
any vertex $a$ in $A(S)$, the arc $\widetilde{f}(a)$ is
independent of the choice of the triangulation $\Delta$ in which
$a$ is exhangeable.
\end{proposition}

 \begin{proof}

Denoting a triangulation as the set of its edges, let
$\Delta=\{a,c_1,\ldots,c_k\}$ and $\Delta'=\{a,c'_1,\ldots,c'_k\}$
be two different triangulations used in the definition of
$\widetilde{f}(a)$.

  Let $R$ be the surface $S$ cut along $a$.  There are two cases for the surface $R$:
  \begin{enumerate}
  \item \label{b1} $R$ has two boundary components
  coming from the curve $a$,
  and in this case there is one distinguished point on each of these boundary
  components, coming from the puncture at the endpoints of $a$. This occurs
  when the arc $a$ joins one puncture of $S$ to itself.
  \item \label{b2}  $R$ has one boundary component, with two distinguished points
  on that boundary. This occurs when $a$ joins two distinct punctures on $S$.
  \end{enumerate}

The triangulations $\Delta$ and $\Delta'$ naturally induce
triangulations on $R$, with the  labelling by edges, $\{a_1,
a_2,c_1,\ldots,c_k\}$ and   $\{a_1, a_2,c'_1,\ldots,c'_k\}$, where
the edge $a$ has been replaced by two edges $a_1$ and $a_2$. In
Case \ref{b1}, each of the edges $a_1$ and $a_2$ appears  on a
boundary component of $R$, and in Case \ref{b2}, the union of the
edges $a_1$ and $a_2$  forms the boundary component of $R$, and
there are two distinguished points on that boundary component.

In each case, we can join the two triangulations induced on $R$ by
$\Delta$ and $\Delta'$ by a finite sequence of elementary moves on
$R$, in such a way that the two distinguished points and the two
edges $a_1$ and $a_2$ on $R$ are left fixed by these elementary
moves. This follows from Harer's result on the connectedness of
the triangulation graph of a surface with boundary and with
distinguished points on boundary components. (See \cite{Harer};
the result is also cited in \cite{Hatcher}.)

Now gluing $a_1$ to $a_2$ back gives a simplicial path in $T(S)$
joining $\Delta$ to $\Delta'$ such that each triangulation in this
path contains $a$ as an edge. We denote the sequence of
triangulations in this path by
\begin{equation}\label{sequence}
\Delta=\Delta_0,\Delta_1,\ldots,\Delta_l=\Delta'.
\end{equation}

 We can assume that this path is simple.

We shall use the following:

\begin{lemma} \label{lem:assume}
We can assume that we can choose the path {\rm (\ref{sequence})}
in such a way that the edge $a$ is exchangeable in each
triangulation representing a vertex of this path.
\end{lemma}

\begin{figure}[hbp]
\psfrag{A}{$\Delta_{i-1}$}
 \psfrag{B}{$\Delta_i$}
  \psfrag{C}{$\Delta_{i+1}$}
   \psfrag{D}{$\Delta'_{i}$}
      \psfrag{E}{$\Delta'_{i+1}$}
     \psfrag{1}{(a)}
    \psfrag{2}{(b)}
\centering
\includegraphics[width=.7\linewidth]{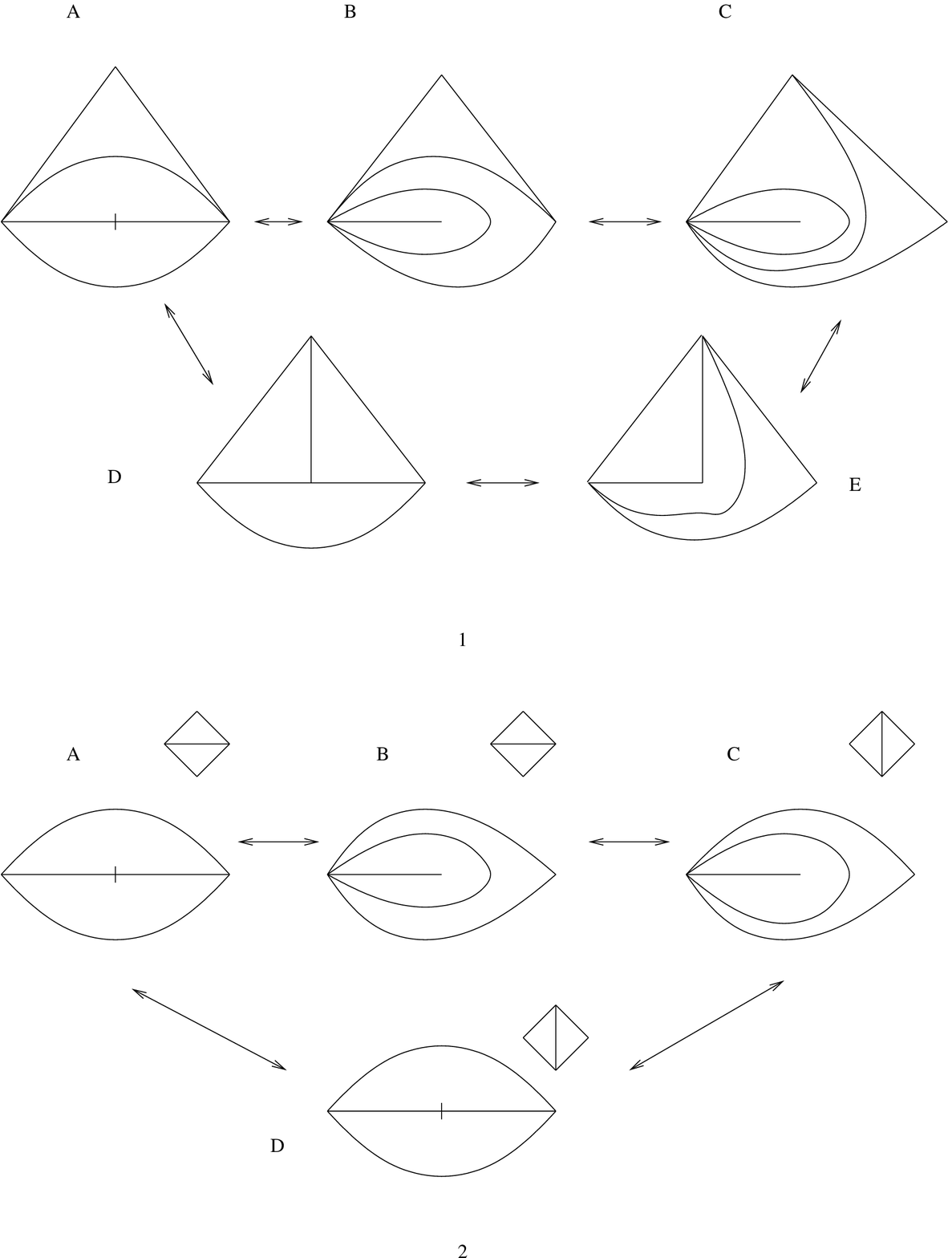}
\caption{\small{Replacing a path in which each vertex contains $a$ by a path
where at each vertex  $a$  is an exchangeable edge.}}
\label{fig:reduce}
\end{figure}

\begin{proof}

Suppose there is a  vertex in the path (\ref{sequence}) that
represents a triangulation in which $a$ is not exchangeable, and
let $\Delta_i$ be the triangulation represented by the first such
vertex, after the vertex $\Delta$. Then, $\Delta_{i+1}$
(respectively $\Delta_{i-1}$) is the triangulation represented by
the vertex after (respectively before) $\Delta_i$. Note that since
$a$ is exchangeable in $\Delta$ and in $\Delta'$, we have $1\leq
i<l$. There are two possibilities for the vertex joining
$\Delta_i$ and $\Delta_{i+1}$, and they are represented
respectively in Figure \ref{fig:reduce} (a) and (b).

The first possibility (top of Figure \ref{fig:reduce} (a)) is when
the elementary move that takes $\Delta_i$ to $\Delta_{i+1}$
involves an edge that is on the boundary of a triangle having the
edge $a^*$ (the dual of $a$) as an edge. (Note that we exclude the
case where the elementary move is performed on the edge $a^*$,
since in that case we recover the triangulation $\Delta_{i-1}$ as
a result, but we assumed the path is simple.) In this case, we
replace the subpath
$\Delta_{i-1}\leftrightarrow\Delta_i\leftrightarrow\Delta_{i+1}$
by the path
$\Delta_{i-1}\leftrightarrow\Delta'_i\leftrightarrow\Delta'_{i+1}\leftrightarrow\Delta_{i+1}$
represented in the bottom of Figure \ref{fig:reduce} (a).

The second possibility (top of Figure \ref{fig:reduce} (b)) is
when the elementary move that takes $\Delta_i$ to $\Delta_{i+1}$
involves an edge that is not on the boundary of a triangle whose
boundary contains $a^*$. In Figure \ref{fig:reduce} (b), we have
symbolically represented that elementary move on a quadrilateral
that is disjoint from $a$ and $a^*$. In that case, we replace the
subpath
$\Delta_{i-1}\leftrightarrow\Delta_i\leftrightarrow\Delta_{i+1}$
by the path
$\Delta_{i-1}\leftrightarrow\Delta'_i\leftrightarrow\Delta_{i+1}$
represented in Figure \ref{fig:reduce} (b).

In each case, each vertex of the new simplicial path that joins
$\Delta$ to $\Delta'$ contains $a$ as an edge, and the number of
occurrences of vertices in this path in which $a$ is
non-exchangeable has been reduced by one. This allows us, by
induction, to obtain a path joining $\Delta$ and $\Delta'$  in
$T(S)$, in which every vertex is represented by a triangulation
that contains $a$ and where $a$ is exchangeable.

 This proves Lemma \ref{lem:assume}

\end{proof}

 We now continue with the proof of Proposition \ref{prop:well}.

Note that the closed paths in $T(S)$ that are represented in
Figure \ref{fig:reduce} (a) and  (b) are a pentagon and a square.

To prove that $\widetilde{f}$ is well-defined, by
Lemma~\ref{lem:assume} it suffices to consider the case where
$\Delta$ and $\Delta'$ are joined by an edge in $T(S)$.

By reordering the edges of the triangulation $\Delta$, we can
assume that $c_1$ is the vertex that is transformed by the
elementary move that takes $\Delta$ to $\Delta'$. We use as before
the notation $\Delta\leftrightarrow \Delta'$, with $\Delta=\langle
a,c_1\rangle $ and $\Delta'=\langle a,c'_1\rangle$, to denote this
elementary move, in which $c'_1$ is the image of $c_1$. Edges
other than $c_1$ are unchanged.
 We distinguish two cases:

\begin{figure}[hbp]
    \psfrag{1}{(a)}
\psfrag{2}{(b)}
 \psfrag{a}{$a$}
\psfrag{x}{$c_1$}
 \psfrag{y}{$c'_1$}
\psfrag{z}{$a'$}
 \psfrag{d}{$d$}
\psfrag{x}{$c_1$}
\centering
\includegraphics[width=.6\linewidth]{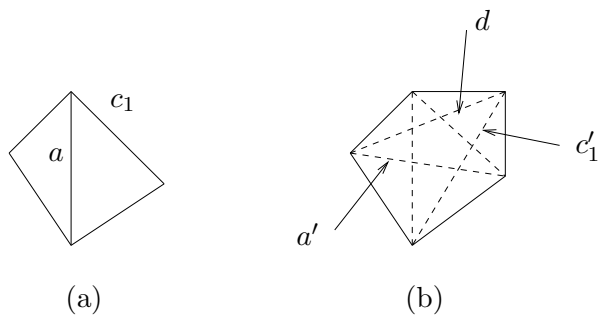}
\caption{\small{The star in the right hand side figure corresponds to a pentagon relation.}}
\label{fig:star}
\end{figure}

   \noindent {\bf Case 1.---}
In the triangulation $\Delta$, $a$ and $c_1$ are two edges of a
common triangle (see Figure \ref{fig:star} (a)). Then, since $S$
is not a torus with one puncture or a sphere with three punctures,
the edge joining $\Delta$ and $\Delta'$ in $T(S)$ belongs to a
pentagon in this graph. This follows from the fact that both $a$
and $c_1$ are exchangeable, and it can be seen diagrammatically in
Figure  \ref{fig:star} (b), where the sequence of elementary moves
representing the pentagon relation is:
  \[\langle a,c_1\rangle \leftrightarrow\langle a,c'_1\rangle,
  \leftrightarrow\langle c'_1,d\rangle, \leftrightarrow\langle d,a'\rangle,
  \leftrightarrow\langle a',c_1\rangle, \leftrightarrow\langle c_1,a\rangle.
  \]

Let $\mathcal{P}$ denote this pentagon. Since $f:T(S)\to T(S)$ is
simplicial, the image $f(\mathcal{P})$ of $\mathcal{P}$ by $f$ is
again a pentagon (Lemma \ref{lem:pentagon}). In the pentagon
$\mathcal{P}$, the triangulation $\langle a,c_1\rangle$ is
labelled $\Delta$, the triangulation  $\langle a',c_1\rangle$ is
labelled $\Delta_a$, the triangulation $\langle a,c'_1\rangle$ is
labelled $\Delta'$ and the triangulation  $\langle d,c'_1\rangle$
is labelled $\Delta'_a$. We then have, by Lemma
\ref{lem:pentagon1} applied to the pentagon $f(\mathcal{P})$,
 \[f(\Delta)-f(\Delta_a)=f(\Delta')-f(\Delta'_a).\]
 This completes the proof of the fact the $\widetilde{f}(a)$ is well-defined in Case 1.

   \noindent {\bf Case 2.---}
In the triangulation $\Delta$, $a$ and $c_1$ are not edges of the
same triangle. In this case, we can perform elementary moves on
$a$ and $c_1$ independently from each other, and we obtain the
following square in  $T(S)$:
 \[\langle a,c_1\rangle \leftrightarrow\langle a,c'_1\rangle ,
 \leftrightarrow\langle a'_1,c'_1\rangle , \leftrightarrow\langle
 a',c_1\rangle , \leftrightarrow\langle a,c\rangle .
 \]
Similarly to the preceding case, the triangulation $\langle
a,c_1\rangle $  is labelled  $\Delta$, the triangulation  $\langle
a',c_1\rangle $ is labelled $\Delta_a$, the triangulation $\langle
a,c'_1\rangle $ is labelled $\Delta'$ and the triangulation
$\langle a'_1,c'_1\rangle $ is labelled $\Delta'_a$, and we get
again, this time using Lemma \ref{lem:square1},
  \[f(\Delta)-f(\Delta_0)=f(\Delta')-f(\Delta'_0).\]

  This shows that  $\widetilde{f}(a)$ is well-defined in each case.

  This completes the proof of Proposition \ref{prop:well}.

  \end{proof}

 The following naturality formula will be useful.
  \begin{proposition}\label{prop:set}
 For any ideal triangulation $\Delta=\{ c_0,c_1,\ldots, c_k\}$, we have
 \[ f(\Delta)=\{\widetilde{f}(c_0),\widetilde{f}(c_1),\ldots,
 \widetilde{f}(c_k)\}.\]
  \end{proposition}
  \begin{proof}
It suffices to prove that
$f(\Delta)\subset\{\widetilde{f}(c_0),\widetilde{f}(c_1),\ldots,
\widetilde{f}(c_k)\}$. Let $a$ be   any (homotopy class of) edge
in $\Delta$. If $a$ is exchangeable, then it follows from the
definition of $\widetilde{f}$   that $\widetilde{f}(a)$ is in the
simplex $f(\Delta)$. Assume now that $a$ is not exchangeable, and
let $a^*$ be its dual edge. Then  $a^*$ is exchangeable in
$\Delta$. Let $\Delta'=\Delta_{a^{*}}$ be the ideal triangulation
obtained from $\Delta$ by exchanging  the edge $a^*$. Finally, let
$\Delta'_a$ be the ideal triangulation obtained from $\Delta'$ by
exchanging  the edge $a$.   Since there is an edge in $T(S)$
joining $\Delta$ to $\Delta'$, and an edge in $T(S)$ joining
$\Delta'$ to $\Delta'_a$, and since $f$ is simplicial,  there is
an edge in $T(S)$ joining $f(\Delta)$ to $f(\Delta')$, and an edge
in $T(S)$ joining $f(\Delta')$ to $f(\Delta'_a)$. By definition,
$\widetilde{f}(a)$ belongs to the triangulation $f(\Delta')$.
Then, the edge $\widetilde{f}(a)$ is transformed by the move
$f(\Delta')\leftrightarrow f(\Delta'_a)$. Therefore,
$\widetilde{f}(a)$ is not an edge of $f(\Delta'_a)$. But since
$\widetilde{f}(a)$ is not transformed by the move
$f(\Delta)\leftrightarrow f(\Delta')$, we conclude that
$\widetilde{f}(a)$ is an edge of $f(\Delta)$.
   \end{proof}

\begin{corollary}\label{cor:simplicial}
The map $\widetilde{f}:A(S)\to A(S)$ is simplicial.
\end{corollary}

\begin{proposition}\label{prop:homo}
If $f$ and $h$ are two automorphisms of $T(S)$, then
$\widetilde{fh}= \widetilde{f} \widetilde{h}.$
\end{proposition}

 \begin{proof}
Let $a$ be (the isotopy class of) an arc on $S$. We show that
$\widetilde{fh} (a)= \widetilde{f} (\widetilde{h} (a))$.

Let $\Delta=\{a,c_1,c_2,\ldots,c_k\}$ be a triangulation in which
$a$ is exchangeable and let $\Delta'=\{a',c_1,c_2,\ldots,c_k\}$ be
the triangulation obtained from $\Delta$ by an elementary move on
$a$. Then, $\widetilde{h}(a)=h(\Delta)-h(\Delta')$, and
$\widetilde{fh}(a)=fh(\Delta)-fh(\Delta')$.

Since the map $h:T(S)\to T(S)$ is simplicial and since the
triangulation $\Delta$ is related to $\Delta'$ in $T(S)$ by an
edge, the triangulation $h(\Delta)$ is also related to the
triangulation $h(\Delta')$ by an edge. Since $h(\Delta)$ contains
the arc $\widetilde{h}(a)$ as an exchangeable arc, we can use
$h(\Delta)$ and  $h(\Delta')$ to define
$\widetilde{f}(\widetilde{h}(a))$, and we obtain:

   \begin{eqnarray*}
 \widetilde{f} (\widetilde{h}(a))
&= &f\left(h(\Delta)\right)-f\left(h(\Delta')\right)
  \\&= &(fh)(\Delta)-(fh)(\Delta')
              \\&=&\widetilde{fh}(a).
     \end{eqnarray*}
          This completes the proof.

 \end{proof}

 \begin{proposition} \label{prop:autom}
For any automorphism $f\in \mathrm{Aut}(T(S))$, the associated map
$\widetilde{f}:A(S)\to A(S)$ is an automorphism.
 \end{proposition}

\begin{proof}
Let $h=f^{-1}\in \mathrm{Aut}(T(S))$ be the inverse of $f$. By
Proposition \ref{prop:homo}, we have $\widetilde{f}\widetilde{h}=
\widetilde{\mathrm{I}}=\widetilde{h} \widetilde{f}$, where
$\mathrm{I}$ is the identity map of $T(S)$. Now from the
definitions, the map $\widetilde{\mathrm{I}}: A(S)\to A(S)$
associated to $\mathrm{I}$ is the identity map of $A(S)$. Thus,
$\widetilde{f}$ has an inverse, which shows that $\widetilde{f}:
A(S)\to A(S)$ is a bijection.

Since $\widetilde{f}$ is also simplicial by
Corollary~\ref{cor:simplicial}, it is an automorphism of $A(S)$.
 \end{proof}

Finally, we prove the following theorem which, together with
Theorem \ref{th:automorphism-arc}, implies Theorem
\ref{th:automorphism}.

 \begin{theorem}
Suppose the surface $S$ is not a sphere with at most three
punctures or a torus with one puncture. Then the map $\phi:
\mathrm{Aut}(T(S))\to  \mathrm{Aut}(A(S))$ defined by $f \mapsto
\widetilde{f}$ is an isomorphism.
  \end{theorem}

 \begin{proof}
Given the propositions that we already proved, it remains to show
is that $\phi$ is bijective.

We first show that $\phi$ is onto. Let $h$ be an arbitrary element
in  $\mathrm{Aut}(A(S))$. By the result of Irmak McCarthy
mentioned above (Theorem \ref{th:automorphism-arc}), $h$ is
induced by a homeomorphism $H:S\to S$, which induces an
automorphism $f$ of $T(S)$, and this automorphism satisfies
$\widetilde{f}=h$. Thus, $\phi$ is onto.

Now we show that $f$ is one-to-one. Let $f$ be an element of
$\mathrm{Aut}(T(S))$ such that $\widetilde{f}$ is the identity
automorphism of $A(S)$. Then, by Proposition  \ref{prop:set}, $f$
acts trivially on $T(S)$. Thus, $f$ is one-to-one. This completes
the proof of the theorem.
 \end{proof}


\begin{thebibliography}{99}

\bibitem{Gromov} M. Gromov, Hyperbolic groups,  In: Essays in Group Theory,
edited by S.M. Gersten. MSRI Publications 8, Springer-Verlag, 1987, 75Ð263.


\bibitem{Harer}   J. L. Harer, Stability of the homology of the mapping class
groups of orientable surfaces, Annals of Math. 121 (l985), 215-249.

\bibitem{Hatcher} A. Hatcher, On triangulations of surfaces.
Top. and its Appl. 40 (1991), 189-194. A new version is available in the
author's webpage.


\bibitem{IK} E. Irmak,  M. Korkmaz,
Automorphisms of the Hatcher-Thurston complex.
 Isr. J. Math. 162, 183-196 (2007).

\bibitem{IM} E. Irmak and J. D. McCarthy,
Injective simplicial maps of the arc complex, Turkish Journal of Mathematics, 33 (2009), 1-16.

\bibitem{ivanov} N. V. Ivanov, Automorphisms of Teichm\"{u}ller modular groups.
Lecture Notes in Math., No. 1346, Springer-Verlag, Berlin and New York, 1988, 199-270.


\bibitem{korkmaz} M. Korkmaz, Automorphisms of complexes of curves on punctured
spheres and on punctured tori. Topology and its Applications  95 no. 2;  85-111 (1999).

 \bibitem{luo} F. Luo,  Automorphisms of the complex of curves,
Topology 39, no. 2, 283-298  (2000).


 \bibitem{Margalit} D. Margalit,
Automorphisms of the pants complex.  Duke Math. J. 121, No. 3, 457-479 (2004).

\end{thebibliography}
\end{document}